\documentclass[12pt]{article}

\usepackage[T1]{fontenc}
\usepackage[active]{srcltx}
\usepackage{lmodern}
\usepackage{amsmath,amstext,amssymb,graphicx,graphics,color}
\usepackage{mathtools,mathrsfs}
\usepackage{enumitem}
\usepackage{tikz}
\usetikzlibrary{decorations.pathmorphing}
\usepackage{mleftright}
\usepackage{theorem}
\usepackage{cite}               
\usepackage{hyperref}           
\usepackage{bbm}                
\usepackage{fancybox}           
\usepackage[section]{placeins}  
\usepackage{mathrsfs}
\usepackage{subcaption}
\allowdisplaybreaks
\theorembodyfont{\normalfont\slshape} 

\newtheorem{theorem}{Theorem}
\newtheorem{proposition}{Proposition}[section]

\newtheorem{lemma}[proposition]{Lemma}
\theoremstyle{break} 

\newenvironment{proof}%
{{\par\noindent \bf Proof. \nobreak}}%
{\nobreak \removelastskip \nobreak \hfill $\Box$ \medbreak}
{{\par\noindent \bf Proof \nobreak}}%
{\nobreak \removelastskip \nobreak \hfill $\Box$ \medbreak}
{{\par\noindent \bf Proof lemma. \nobreak}}%
{\nobreak \removelastskip \nobreak \bf End proof lemma. \medbreak}
\newenvironment{remark}{\par \medskip \noindent {\bf Remark. }\nobreak}{\par \medskip}
\usepackage{fancyhdr}
 \fancypagestyle{plain}{
  \fancyhead{}
  
}

\def\paragraph#1{{\bf #1\ }}
\newcommand{\RN}[1]{%
  \textup{\uppercase\expandafter{\romannumeral#1}}%
}
\newcommand{\expo}{\mathrm{e}}

\newcommand{\dd}{\mathrm{d}}

\newcommand{\emp}{\mathrm{emp}}

\newcommand{\HH}{\mathrm{H}}
\newcommand{\overbar}[1]{\mkern 1.5mu\overline{\mkern-1.5mu#1\mkern-1.5mu}\mkern 1.5mu}
\renewcommand{\RN}[1]{%
  \textup{\uppercase\expandafter{\romannumeral#1}}%
}
\DeclareMathOperator*{\argmin}{\arg\!\min}
\usepackage[utf8]{inputenc}
\usepackage{unicode_character}

\def\Proof{\noindent{\bf Proof}\quad}
\def\qed{\hfill$\square$\smallskip}

\title{Wealth exchange under ceiling and flooring constraints: a modified Bennati–Dragulescu–Yakovenko model}

\author{Fei Cao \footnotemark[1] \and Sebastien Motsch \footnotemark[2] \and Wendy Garcia Umbarita \footnotemark[2]}

\begin{document}
\maketitle

\footnotetext[1]{Amherst College - Department of Mathematics, Amherst, MA 01002, USA}
\footnotetext[2]{Arizona State University - School of Mathematical and Statistical Sciences, 900 S Palm Walk, Tempe, AZ 85287, USA}

\tableofcontents

\begin{abstract}
We investigate the classical Bennati-Dragulescu-Yakovenko (BDY) dollar exchange model introduced in \cite{dragulescu_statistical_2000} where the effects of wealth ceiling and wealth flooring are explored. In our model, $N$ identical economical agents involved in the BDY game are also subjected to certain policies issued by a (artificial) government, which prevent agents whose wealth exceeds some prescribed threshold value (denoted by $b \in \mathbb N_+$) from receiving money and which prohibit agents whose wealth falls below certain threshold value (denoted by $a \in \mathbb N$) from giving out their money. We derive a mean-field system of coupled nonlinear ordinary differential equations (ODEs) governing the evolution of the distribution of money as the number of agents $N$ tends to infinity and study the large time behavior of the resulting ODE system. The impact of a wealth cap and a wealth floor on economic inequality (measured by the Gini index) will also be explored numerically.
\end{abstract}


\section{Introduction}\label{sec:1}
\setcounter{equation}{0}

In this work, we study a variant of the classical Bennati-Dragulescu-Yakovenko (BDY) dollar exchange model introduced in \cite{dragulescu_statistical_2000}. In the original BDY model, there are $N$ identity agents labeled by $1$ through $N$ and each of them is characterized by the amount of dollars he/she has. We denote by $S_i(t)$ the amount of dollars agent $i$ has at time $t$. The game is a simple mechanism for dollar exchange taking place in a closed economical system/market, where at random times (generated by an exponential law) an agent $i$ picked uniformly at random gives a dollar (if he/she has at least one dollar) to another agent $j$ again picked uniformly at random, and if agent $i$ is ruined (i.e., $S_i= 0$) then nothing happens. We can represent the BDY game as follows:
\begin{equation}\label{dynamics:BDY}
\textbf{BDY game:} \qquad (S_i,S_j)~ \begin{tikzpicture} \draw [->,decorate,decoration={snake,amplitude=.4mm,segment length=2mm,post length=1mm}]
(0,0) -- (.6,0); \node[above,red] at (0.3,0) {};\end{tikzpicture}~  (S_i-1,S_j+1) \quad (\text{if } S_i\geq 1).
\end{equation}
Since the economical system is closed, we must have
\begin{equation}\label{eq:preserved_sum}
S_1(t)+ \cdots +S_N(t) = N\mu \qquad \text{for all } t\geq 0.
\end{equation}
for some $\mu > 0$ which represents the average wealth per agent. For simplicity we will enforce $\mu \in \mathbb N_+$ to be a fixed positive integer. The BDY model described above is one of the earliest models in econophysics and has been studied extensively across different communities since its inception \cite{cao_derivation_2021,cao_interacting_2022,lanchier_rigorous_2017,merle_cutoff_2019}. Due to the fact that all agents with at least one dollar give to the rest of agents at a fixed rate and the game is biased towards any specific agent (or certain group of agents), the BDY model is termed the \emph{unbiased exchange model} in \cite{cao_derivation_2021,cao_interacting_2022} and the \emph{one-coin model} in \cite{lanchier_rigorous_2017}. Subsequent extensions of the basic BDY game suggest the presence of a bank (or even multiple banks) which allows agents to be indebted (i.e., $S_i < 0$), and we refer the interested readers to a series of recent works \cite{cao_bias_2023,cao_uncovering_2022,lanchier_rigorous_2019,lanchier_distribution_2022}. Other natural extensions of the BDY dynamics include the introduction of bias, modeled for instance through taxation or redistribution mechanisms \cite{cao_derivation_2021,cao_uniform_2024,miao_nonequilibrium_2023}, which favor either poorer or richer agents in pairwise exchanges. Another possible variation introduces probabilistic cheaters \cite{blom_hallmarks_2024,cao_mean_2025}, in which agents may be honest or with positive probability misrepresent their wealth by claiming to have no money to give, thereby securing an advantage through concealment. Lastly, a continuous and infinitesimal version of the classical BDY model, in the form of a nonlinear partial differential equation coupled with a nonlinear Robin-type boundary condition, has been derived and analyzed in a very recent work \cite{cao_bennati_2025}.

In this manuscript, we aim to introduce and investigate a version of the standard BDY model which incorporates a wealth cap and a wealth floor for each agent. For this purpose, we introduce two additional model parameters $b \in \mathbb N_+$ and $a \in \mathbb{N}$ such that
\begin{equation}\label{eq:condition}
a < \mu < b.
\end{equation}
Our modified agent-based BDY model can be summarized as follows:
\begin{equation}\label{dynamics:modified_BDY}
\textbf{Modified BDY:} \qquad (S_i,S_j)~ \begin{tikzpicture} \draw [->,decorate,decoration={snake,amplitude=.4mm,segment length=2mm,post length=1mm}]
(0,0) -- (.6,0); \node[above,red] at (0.3,0) {};\end{tikzpicture}~  (S_i-1,S_j+1) \quad (\text{if } S_i > a ~\text{and}~ S_j < b  ).
\end{equation}
In other words, $a$ and $b$ can be viewed, respectively, as a wealth floor which prevents agents with insufficient wealth from giving a dollar, and a wealth cap which prevents very wealthy agents from receiving one as illustrated in Figure \ref{fig:illustration_model}. The introduction of the parameters $a$ and $b$ can be interpreted as a government policy designed to mitigate economic inequality among agents, as quantified by the Gini index. From now on, we classify agents according to their wealth: agent $i$ is said to be poor (lower class) if $S_i \leq a$, rich (upper class) if $S_i \geq b$, and middle class if $a < S_i < b$. Therefore, the modified BDY dynamics \eqref{dynamics:modified_BDY} introduced in this work partition agents into three classes (based solely on their wealth), hence we term the model as a \textbf{three-classes BDY model}. It is worth emphasizing that our model boils down to the classical BDY model \eqref{dynamics:BDY} in the particular case when $a = 0$ and $b \to +\infty$.

\begin{figure}[ht]
    \centering
    \includegraphics[width=.9\textwidth]{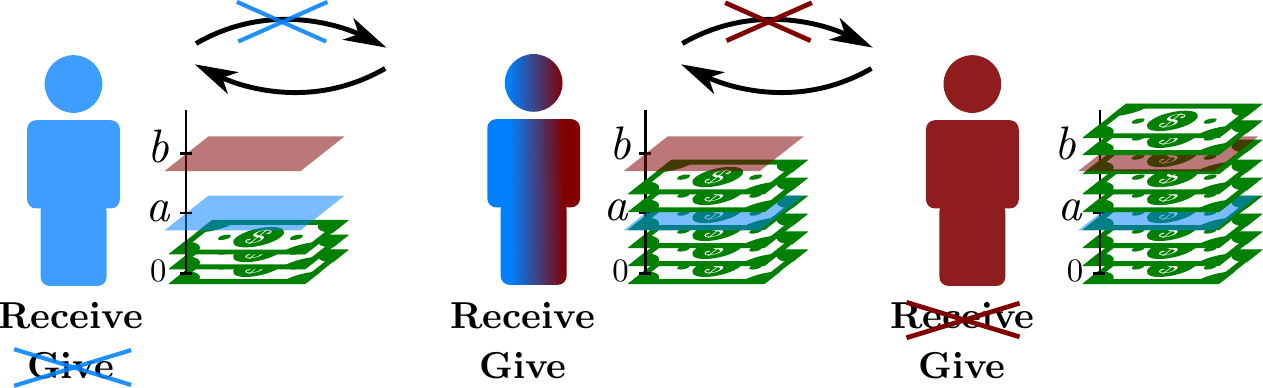}
    \caption{In the proposed binary exchange dynamics, an agent with less than $a$ dollars ({\it poor} agent) represented on the left, can no longer give. Similarly, an agent with more than $b$ dollars ({\it rich} agent) represented on the right, cannot receive. The middle class consists of agents whose wealth lies between $a$ and $b$ (represented in the middle), they can both give and receive money.}
    \label{fig:illustration_model}
\end{figure}

To the best of our knowledge, this is the first work to incorporate both a wealth cap and a wealth floor into an econophysics model and to analyze it from a mathematical point of view. From a policy-making perspective, the concept of Universal Basic Income (UBI) has recently gained significant attention, driven by increasing concerns over job displacement resulting from the deployment of advanced AI technologies. Thus, our work can also be viewed as a first step toward exploring the impact of specific policies on the overall wealth distribution in an artificial society.

Although the present work focuses on a particular binary exchange model, alternative trading mechanisms can be formulated and analyzed, giving rise to distinct dynamics. For instance, in the immediate exchange model \cite{heinsalu_kinetic_2014}, pairs of agents randomly exchange independent and uniformly distributed fractions of their wealth. Another well-studied variant is the uniform reshuffling model \cite{dragulescu_statistical_2000,cao_entropy_2021,matthes_steady_2008} and its variants \cite{cao_binomial_2024,chakraborti_statistical_2000,chatterjee_pareto_2004}, in which the combined wealth of two randomly chosen agents before interaction is uniformly redistributed between them afterward. Additional models originating from econophysics can be found in \cite{during_kinetic_2008,pareschi_interacting_2013,pereira_econophysics_2017,savoiu_econophysics_2013}.

The remainder of this manuscript is structured as follows. In section \ref{sec:2}, we derive the mean-field system of nonlinear ODEs corresponding to the agent-based model \eqref{dynamics:modified_BDY} and identify its unique equilibrium distribution. Section \ref{sec:3} is devoted to analyzing the large time behavior of the mean-field ODE system, where we establish convergence of solutions to the unique equilibrium via the introduction of a novel generalized entropy functional. Section \ref{sec:4} concludes the paper, and Appendix \ref{sec:Appendix} contains the proofs of several auxiliary results.

\section{Mean-field limit}\label{sec:2}
\setcounter{equation}{0}

\subsection{Derivation of the mean-field ODE system}

Let ${\bf p}(t)=\left(p_0(t),p_1(t),\ldots\right)$ be a probability mass function, where $p_n(t)$ represents the (deterministic) proportion of agents with $n$ dollars at time $t$ in the large population limit $N \to \infty$. In other words, $p_n(t) = \lim_{N \to \infty} \frac{1}{N}\,\sum_{i=1}^N \mathbbm{1}_{\{S_i(t) = n\}}$ (where the notion of convergence will be made precise in Theorem \ref{thm:PoC} below). It is natural, from a mean-field perspective, to expect that the evolution of ${\bf p}(t)$ will be governed by the following coupled system of nonlinear ODEs (see figure \ref{fig:illustration_mean_field}:
\begin{equation}\label{eq:law_limit}
\frac{\dd}{\dd t} {\bf p}(t) = \mathcal{L}[{\bf p}(t)],
\end{equation}
where
\begin{equation}\label{eq:L}
\mathcal{L}[{\bf p}]_n \coloneqq \lambda_r\, (\mathbbm{1}_{\{a \leq n\}}\,p_{n+1}  - \mathbbm{1}_{\{a+1 \leq n \}}\, p_{n})
      + \lambda_g\,(\mathbbm{1}_{\{1\leq n\leq b\}}\, p_{n-1} - \mathbbm{1}_{\{ n \leq b-1\}}\, p_{n})
\end{equation}
\bigskip
and
\begin{equation}\label{eq:two-classes}
\lambda_r \coloneqq \sum\limits_{0\leq \ell \leq b-1} p_\ell,~~ \lambda_g \coloneqq \sum\limits_{\ell \geq a+1} p_\ell
\end{equation}

\begin{figure}[ht]
    \centering
    \includegraphics[width=.8\textwidth]{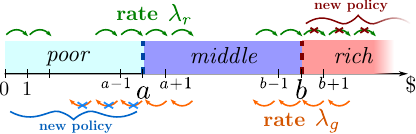}
    \caption{Illustration of the wealth exchange dynamics in the mean-field limit \eqref{eq:law_limit}-\eqref{eq:two-classes} as $N \to \infty$. The population is described through a probability mass function ${\bf p}=(p_0,p_1,\dots)$ whose evolution is prescribed by the rates $λ_g$ and $λ_r$, which represent the proportion of givers and receivers, respectively.}
    \label{fig:illustration_mean_field}
\end{figure}

We emphasize that in the special case where $a = 0$ and $b \to \infty$, the system of nonlinear ODEs \eqref{eq:law_limit}-\eqref{eq:two-classes} coincides with the mean-field ODE system for the usual BDY model. On the other hand, the passage from the stochastic $N$-agent dynamics \eqref{dynamics:modified_BDY} to the infinite system of ODEs \eqref{eq:law_limit} in the large population limit $N \to \infty$ relies on the concept of propagation of chaos \cite{chaintron_propagation_2022,chaintron_propagation_2022_partII,sznitman_topics_1991}, which has been rigorously established for a variety of models from socio-economic sciences (including the classical BDY dynamics)\cite{cao_derivation_2021,cao_entropy_2021,cao_explicit_2021,cao_fractal_2024,cao_interacting_2022,cao_uniform_2024} as well as physical and biological sciences \cite{cortez_quantitative_2016,cortez_uniform_2016,carlen_kinetic_2013}.

To show the convergence of the agent-based dynamics \eqref{dynamics:modified_BDY} to its mean-field dynamical system \eqref{eq:law_limit}-\eqref{eq:two-classes} as the number of agents $N$ tends to infinity, we introduce the empirical distribution associated to our modified BDY model:
\begin{equation}\label{eq:rho_emp}
{\bf p}_{\emp}(t) \coloneqq \frac{1}{N}\,\sum_{n=1}^N \delta_{S_n(t)},
\end{equation}
where $\delta$ denotes a Dirac delta distribution. In addition, for any continuous and bounded test function $\varphi$ we denote $\langle {\bf p}_{\emp}, \varphi \rangle \coloneqq \frac{1}{N}\,\sum_{n=1}^N \varphi(S_i)$. We can prove the following version of the propagation of chaos property:

\begin{theorem}\label{thm:PoC}
Assume that ${\bf p}(t)=\{p_n(t)\}_{n\geq 0}$ is a classical solution of
\eqref{eq:law_limit}-\eqref{eq:two-classes} such that ${\bf p}(0) \in
\mathcal{P}(\mathbb N)$ with mean $\mu \in \mathbb N_+$, and ${\bf p}_{\emp}(t)$
represents the empirical distribution associated to the modified BDY model
\eqref{dynamics:modified_BDY}. Then there exists some universal constant
$C > 0$ such that
\begin{equation}\label{eq:PoC}
\mathbb{E}\left[\|{\bf p}(t) - {\bf p}_{\emp}(t)\|_{\ell^1}\right] \leq \left[\mathbb{E}\left[\|{\bf p}(0) - {\bf p}_{\emp}(0)\|_{\ell^1}\right] + \frac{C\,t\,\sqrt{\log N}}{\sqrt{N}}\right]\,\expo^{8t}
\end{equation}
for all $t\geq 0$. In particular, if $\mathbb{E}\left[\|{\bf p}(0) - {\bf p}_{\emp}(0)\|_{\ell^1}\right] \xrightarrow{N \to 0} 0$, then for any $T > 0$, $\mathbb{E}\left[\|{\bf p}(t) - {\bf p}_{\emp}(t)\|_{\ell^1}\right] \xrightarrow{N \to 0} 0$ for each $t\in [0,T]$.
\end{theorem}

\begin{remark}
Theorem~\ref{thm:PoC} ensures that, in the large population limit $N \to \infty$, the empirical law ${\bf p}_{\emp}(t)$ converges to the deterministic mean-field law ${\bf p}(t)$ in expected total variation distance on any finite time interval.
\end{remark}

We emphasize that Theorem~\ref{thm:PoC} can be proved by a martingale-based technique similar to those employed in \cite{cao_derivation_2021,merle_cutoff_2019}. For the sake of completeness, we provide a self-contained proof in Appendix~\ref{sec:Appendix}. Here, we illustrate numerically how the stochastic agent-based system approaches its deterministic mean-field limit as the population size $N$ increases (for a fixed time).

\begin{figure}[ht]
    \centering
    \includegraphics[width=.58\textwidth]{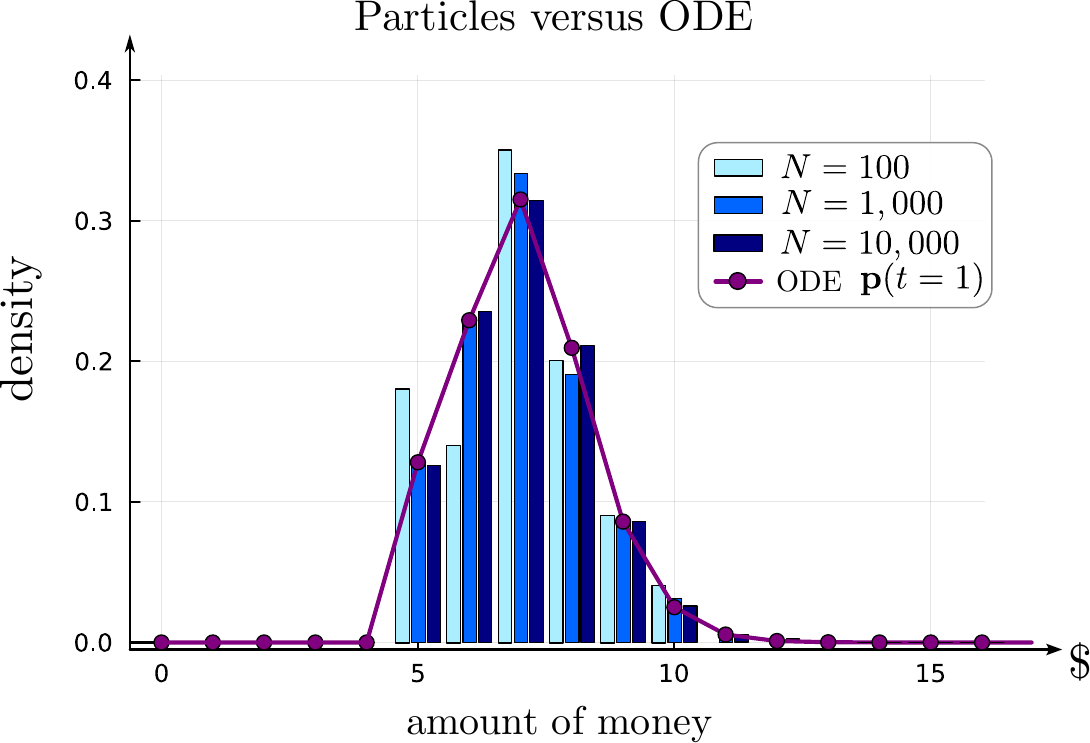}
    \includegraphics[width=.39\textwidth]{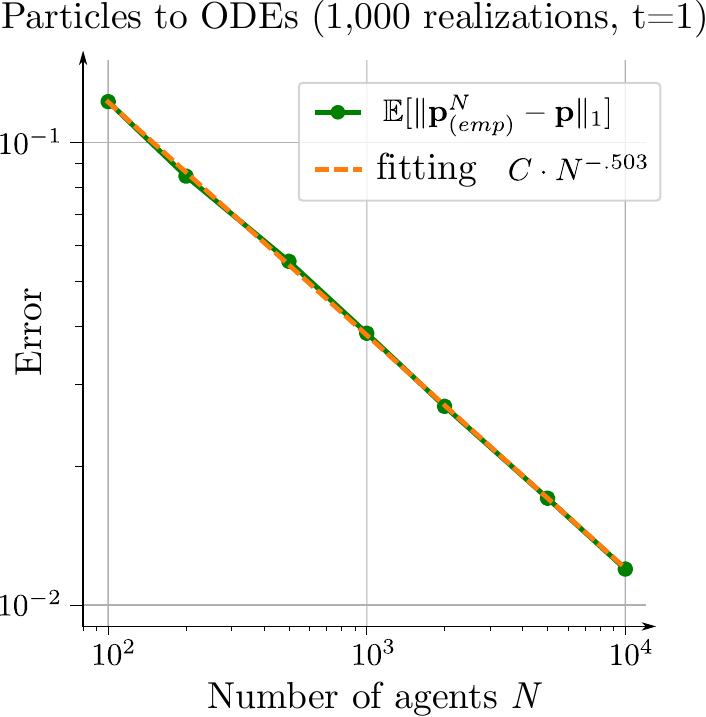}
    \caption{\textbf{Left:} Empirical wealth distributions obtained from particle simulations at time $t = 1$ for different $N$-agents, alongside the mean field ODE solution (solid line). \textbf{Right:} Error between particle distribution for different $N$-agents and ODE solution at $t = 1$ (with the plot being in the log-log scale).}
    \label{fig:Nlimit}
\end{figure}

\subsection{Elementary properties of the solution}

We present some elementary properties of solutions to the mean-field ODE system \eqref{eq:law_limit}-\eqref{eq:two-classes}.
\begin{lemma}\label{lem:invariant}
Assume that ${\bf p}(t)=\{p_n(t)\}_{n\geq 0}$ is a classical solution of \eqref{eq:law_limit}-\eqref{eq:two-classes} such that ${\bf p}(0) \in \mathcal{P}(\mathbb N)$ with mean $\mu \in \mathbb N_+$, then
\begin{equation}\label{eq:conservation_mass_mean_value}
\sum_{n=0}^\infty \mathcal{L}[{\bf p}]_n =0 \quad \textrm{and} \quad \sum_{n=0}^\infty n\,\mathcal{L}[{\bf p}]_n =0.
\end{equation}
In particular, the total probability mass and the mean value are preserved over time.
\end{lemma}

\Proof Direct calculations lead us to
\begin{equation*}
\sum_{n=0}^\infty \mathcal{L}[{\bf p}]_n =\sum_{n\geq a} p_{n+1}\,\lambda_r  - \sum_{n\geq a+1} p_n\,\lambda_r + \sum_{1 \leq n \leq b}p_{n-1}\,\lambda_g - \sum_{0 \leq n \leq b-1} p_n\,\lambda_g = 0\\
\end{equation*}
Thus, the total probability mass is conversed. Multiplying the $n$-th component of $\mathcal{L}[{\bf p}]$ by $n$ and summing over $n \in \mathbb N$ yields
\begin{align*}
\sum_{n=0}^\infty n\,\mathcal{L}[{\bf p}]_n &= \lambda_r\sum_{n\geq a} n\,p_{n+1} - \lambda_r\sum_{n\geq a+1} n\,p_{n} + \lambda_g\sum_{1 \leq n \leq b} n\,p_{n-1}- \lambda_g\sum_{0 \leq n \leq b-1} n\,p_{n} \\
&= -\lambda_r\sum\limits_{n \geq a+1} p_n + \lambda_g\sum\limits_{0\leq n\leq b-1} p_n \\
&= -\lambda_r\lambda_g + \lambda_g\lambda_r = 0.
\end{align*}
Thus the proof is completed. \qed

Thanks to Lemma \ref{lem:invariant}, we deduce that ${\bf p}(t) \in V_\mu$ for all $t\geq 0$, where
\begin{equation}\label{eq:space}
V_\mu \coloneqq \left\{{\bf p} \mid \sum_{n=0}^\infty p_n =1,~p_n \geq 0,~\sum_{n=0}^\infty n\,p_n =\mu\right\}
\end{equation}
is the space of probability mass functions with the given mean value $\mu$.

In order to identify the equilibrium distribution, denoted by ${\bf p}^*$, of the system of nonlinear ODEs \eqref{eq:law_limit}-\eqref{eq:two-classes}, we set $\mathcal{L}[{\bf p}^*]_n = 0$. Since the mean value of equilibrium distribution is $\mu$ and satisfies $a<\mu<b$, the rates $\lambda_r$ and $\lambda_g$ are strictly positive for all times.

A candidate for the equilibrium distribution is given by the formula: 
\begin{equation}\label{eq:equi}
    p^*_n =   \left\{
    \begin{array}{ll}
        0 &\quad \text{if } 0 \leq n < a, \\
        \overbar{r}^{n-a}\,p^*_a &\quad \text{if } a \leq n \leq b,\\
        0 &\quad \text{if } n > b,
    \end{array}
    \right.
\end{equation}
where $\overbar{r} \in \mathbb{R}_+$ and $p_a^*$ are chosen such that ${\bf p}^* \in V_\mu$ is a probability mass function with the prescribed mean $\mu$. We refer to the number $\overbar{r}$ as the \emph{common ratio}. Heuristically speaking, ${\bf p}^*$ is a truncated geometric-type distribution supported on a possibly bounded domain. 
To determine $\overbar{r}$ and $p_a^*$, we first express $p_a^*$ as a function of $\overbar{r}$ using the condition that ${\bf p}^*$ is a probability mass on $\mathbb N$:
\[1=\sum\limits_{n\geq 0} p^*_n = \sum\limits_{n=a}^b p^*_n = p^*_a\,\frac{1-\overbar{r}^{b+1-a}}{1-\overbar{r}}, \]
which leads us to
\begin{equation}
\label{eq:p_star_r}
p^*_a = \frac{1-\overbar{r}}{1-\overbar{r}^{b+1-a}}.
\end{equation}
To find the common ratio $\overbar{r}$, we rely on the fact that the mean value of the distribution ${\bf p}^*$ is $\mu$:
\begin{align*}
\mu = \sum\limits_{n\geq 0} n\,p^*_n &= 
\frac{p^*_a}{\overbar{r}^a}\,\sum\limits_{n=a}^b n\,\overbar{r}^n\\
&= \frac{p^*_a}{\overbar{r}^a}\,\frac{\overbar{r}}{(1-\overbar{r})^2}\left[a\,\overbar{r}^{a-1} +
(1-a)\,\overbar{r}^a-(b+1)\,\overbar{r}^b+ b\,\overbar{r}^{b+1}\right]
\end{align*}
Employing \eqref{eq:p_star_r} to replace $p_a^*$ yields
\begin{align*}
\mu &=\frac{1-\overbar{r}}{\overbar{r}^a\,(1-\overbar{r}^{b+1-a})}\,\frac{\overbar{r}}{(1-\overbar{r})^2}\left[a\,\overbar{r}^{a-1} +
(1-a)\,\overbar{r}^a-(b+1)\,\overbar{r}^b+ b\,\overbar{r}^{b+1}\right].
\end{align*}
Consequently, we deduce that $\overbar{r} \in \mathbb{R}_+$ must satisfy
\begin{equation}\label{eq:poly}
(b-\mu)\,\overbar{r}^{b+2-a} + (\mu-b-1)\,\overbar{r}^{b+1-a} + (\mu-a+1)\,\overbar{r} + a - \mu = 0.
\end{equation}
We now state a technical result regarding the (positive) roots of the polynomial defined by the left hand side of \eqref{eq:poly}.
\begin{proposition}\label{prop:polyroot}
Given $0 \leq a < \mu < b$ with $a, \mu, b \in \mathbb N$, and define for $x \geq 0$ the polynomial function
\begin{equation}\label{eq:f}
f(x) =  (b-\mu)\,x^{b+2-a} + (\mu-b-1)\,x^{b+1-a} + (\mu-a+1)\,x + a - \mu.
\end{equation}
Then \begin{enumerate}[label=(\roman*)]
\item when $\mu<\frac{a+b}{2}$, $f(x)$ has a unique root within $(0,1)$,
\item when $\mu=\frac{a+b}{2}$, $f(x)$ has a unique root at $x= 1$,
\item when $\mu>\frac{a+b}{2}$, $f(x)$ has a unique root within $(1,\infty)$.
\end{enumerate}
\end{proposition}

\medskip
\begin{figure}[ht]
    \centering
    \includegraphics[width=.99\textwidth]{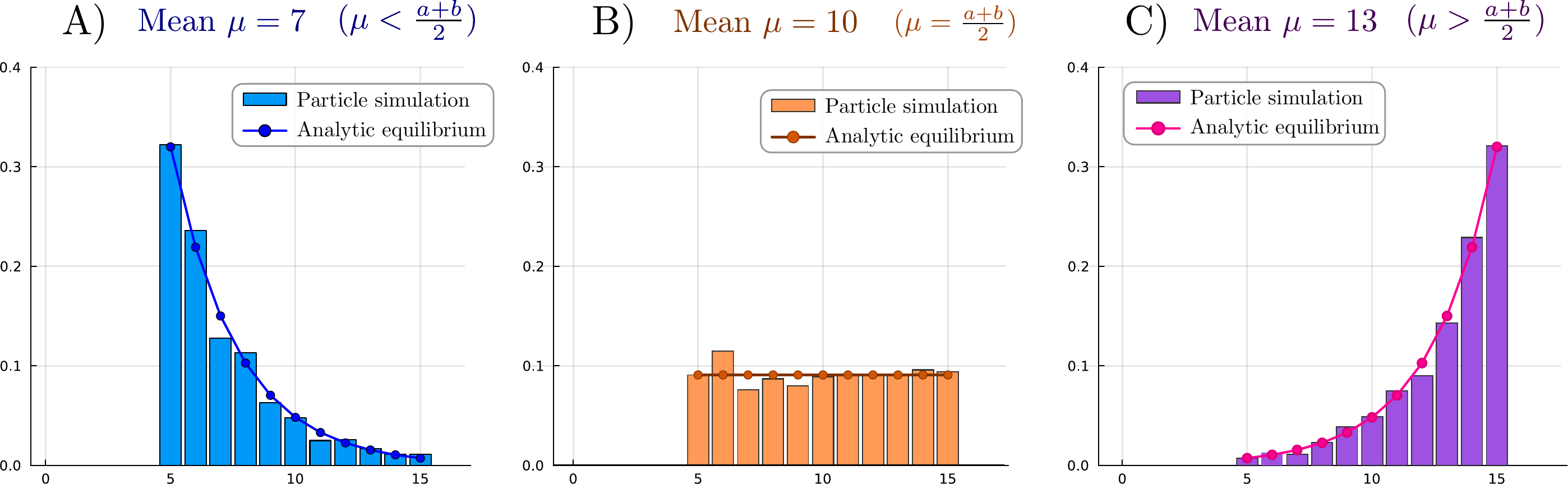}
    \caption{Geometric-type Equilibrium distribution ${\bf p}^*$ \refeq{eq:equi} for different values of the mean value $\mu$. \textbf{A)} Corresponds to the case where the root resides within $(0,1)$ and the resulting distribution is decreasing within $[a,b]$.
             \textbf{B)} Corresponds to the case where the root is exactly at $x = 1$ and the distribution is (nearly) uniform within $[a,b]$.
             \textbf{C)} Corresponds to the case where the root lies within $(1,\infty)$ and the distribution is increasing within $[a,b]$.}
    \label{fig:group-equilibrium}
\end{figure}

\medskip

\Proof We start with the proof of (ii) and hence assume that $b-\mu = \mu-a$. In this case, the equation $f(x) = 0$ implies
\begin{equation}\label{eq:casei}
(\mu-a)\,x^{2\,(\mu-a+1)} - (\mu-a+1)\,x^{2\,(\mu-a)+1} + (\mu-a+1)\,x + a - \mu = 0.
\end{equation}
Factorizing the left hand side of \eqref{eq:casei} gives rise to
\[(x-1)^3\,\sum\limits_{\ell=1}^{\mu-a}\left[\ell\,x^{\mu-a+\ell-1} + \ell\,x^{\mu-a+\ell-2} + \cdots + \ell\,x^\ell + \ell\,x^{\ell-1} + (\ell-1)\,x^{\ell-2} + \cdots + 1 \right] = 0,\]
from which it follows readily that the only positive root of $f(x)$ occurs at $x = 1$ (as the polynomial multiplying the factor $(x-1)^3$ is obviously positive for all $x \geq 0$). Now we turn to the proof of part (i) and assume that $b-\mu > \mu-a$ (the proof of part (iii) resembles the proof of part (i) and will be omitted). We decompose the polynomial $f$ as follows:
\[f(x) = (x-1)^2\,g(x) ~~\textrm{with}~~ g(x) = x^{\mu-a}\,\sum\limits_{\ell=1}^{b-\mu} \ell\,x^\ell - \sum\limits_{\ell=1}^{\mu-a}\left(1+x+\cdots+x^{\ell-1}\right).\] Since $0\leq a < \mu < b$ and $b-\mu > \mu-a$, we have $g(0) = a - \mu < 0$ and $g(1) = \sum\limits_{\ell=1}^{b-\mu} \ell - \sum\limits_{\ell=1}^{\mu-a} \ell > 0$, the mean value theorem guarantees the existence of some $\overbar{x} \in (0,1)$ for which $g(\overbar{x}) = 0$. Moreover, the derivative of $g$ is easily computable and \[g'(x) = \sum\limits_{\ell=1}^{b-\mu} (\mu-a+\ell)\,\ell\,x^{\mu-a+\ell-1} - \sum\limits_{\ell=1}^{\mu-a} \left[(\ell-1)\,x^{\ell-2} + \cdots + 2\,x + 1\right]. \] We observe that $g'(x) > 0$ for all $x > 1$ since
\[g'(x) > \sum\limits_{\ell=1}^{b-\mu} \ell\,(\underbrace{x^\ell + x^{\ell} + \cdots x^{\ell}}_{\ell\rm\ times}) - \sum\limits_{\ell=1}^{\mu-a} \ell\,\left(x^{\ell-1} + \cdots +x+1\right) > 0 \] for all $x\geq 1$. Consequently, the polynomial $g(x)$ has no roots within $(1,\infty)$. In order to prove that $g$ has no other roots within $(0,1)$ (other than $\overbar{x}$), it suffices to show that
\begin{equation}\label{eq:sufficient}
g'(x_*) > 0
\end{equation}
whenever $x_* \in (0,1)$ is a root of $g$. Indeed, suppose to the contrary (and without loss of generality) that $x_1 < x_2$ are two distinct roots of $g$ within $(0,1)$ such that there are no other roots of $g$ within $(x_1,x_2)$. As $g'(x_1) > 0$, we deduce that $g(x) > 0$ for all $x \in (x_1,x_2)$. On the other hand, by continuity and the fact that $g'(x_2) > 0$, we also obtain $g'(\xi) > 0$ for any $\xi$ sufficiently close to $x_2$, which contradicts the equation $g(x_2) = 0$. Finally, it remains to prove \eqref{eq:sufficient}. Since $x_* \in (0,1)$ is a root of $g$, we can perform the following estimates:
\begin{align*}
g'(x_*) &= \sum\limits_{\ell=1}^{b-\mu} \frac{(\mu-a+\ell)}{x_*}\,\ell\,x^{\mu-a+\ell}_* - \sum\limits_{\ell=1}^{\mu-a} \left[(\ell-1)\,x^{\ell-2}_* + \cdots + 2\,x_* + 1\right] \\
& > \sum\limits_{\ell=1}^{b-\mu} \frac{(\mu-a+\ell)}{x_*}\,\ell\,x^{\mu-a+\ell}_* - (\mu-a)\,\sum\limits_{\ell=1}^{\mu-a} \left[x^{\ell-1}_* + \cdots + x_* + 1\right] \\
& > (\mu-a)\,g(x_*) = 0,
\end{align*}
from which the advertised inequality follows. \qed

As an immediate corollary thanks to Proposition \ref{prop:polyroot}, we deduce that the common ratio $\overbar{r}\in \mathbb{R}_+$ in the definition of the truncated geometric-type equilibrium distribution ${\bf p}^*$ \eqref{eq:equi} satisfies
\begin{itemize}
\item if $\mu<\frac{a+b}{2}$, then $\overbar{r} < 1$ and ${\bf p}^*$  is a decaying function on $[a,b]$ (see Figure \ref{fig:group-equilibrium}-A)
\item if $\mu=\frac{a+b}{2}$, then $\overbar{r} = 1$ and ${\bf p}^*$  is flat on $[a,b]$ (see Figure \ref{fig:group-equilibrium}-B)
\item if $\mu>\frac{a+b}{2}$, then $\overbar{r} > 1$ and ${\bf p}^*$  is increasing on $[a,b]$ (see Figure \ref{fig:group-equilibrium}-C).
\end{itemize}

\begin{remark}
In general, a closed-form expression for the common ratio $\overbar{r}$ in terms of $a$, $b$, and $\mu$ does not appear to be attainable at least when $\mu \neq (a+b)/2$. However, in the special case where $b \to \infty$ and $a\in \mathbb N$, we can determine $\overbar{r}$ explicitly and in this case the geometric-type equilibrium distribution ${\bf p}^*$ boils down to
\begin{equation}\label{eq:equi_b=infty}
p^*_n = 0 ~~\text{for}~~ n\leq a-1,~~  p^*_a = \frac{1}{\mu-a+1},~~p^*_n = \left(\frac{\mu-a}{\mu-a+1}\right)^{n-a}\,p^*_a ~~\text{for}~~ n\geq a.
\end{equation}
\end{remark}

\section{Large time behavior}\label{sec:3}
\setcounter{equation}{0}

\subsection{Convergence to equilibrium}\label{subsec:3.1}
In this subsection, we turn our attention to the problem of showing convergence of solutions of the nonlinear ODE system \eqref{eq:law_limit}-\eqref{eq:two-classes} to its unique equilibrium ${\bf p}^*$ \eqref{eq:equi}. The key ingredient on which we will rely is the construction of an appropriate Lyapunov functional associated with the evolution equation \eqref{eq:law_limit}-\eqref{eq:two-classes}. We will restrict ourselves to the set of admissible initial datum ${\bf p}(0)$ satisfying
\begin{equation}\label{eq:IC_condition}
{\bf p}(0) \in \mathcal{A} \coloneqq \left\{{\bf p} \in V_\mu \mid \sum\limits_{n\geq 0} 2^n\,p_n < \infty \right\}.
\end{equation}
We remark here that the constant $2$ appearing in the assumption \eqref{eq:IC_condition} can be replaced by any constant $K$ as long as $K > 1$, but for ease of presentation, we stick to the choice of $K=2$. Also, the restriction \eqref{eq:IC_condition} on the initial data is mild from a practical point of view, since one has to perform a truncation step in order to implement numerical experiments related to the mean-field ODE system \eqref{eq:law_limit}-\eqref{eq:two-classes} of infinite dimension.

We recall from earlier work \cite{cao_derivation_2021,cao_interacting_2022,dragulescu_statistical_2000} on the classical BDY model, which corresponds to taking $a=0$ and $b \to \infty$ in this manuscript, that the Boltzmann entropy (also known as the Boltzmann's $\HH$ functional)
\begin{equation}\label{eq:H_functional}
\HH[{\bf p}] \coloneqq \sum\limits_{n\geq 0} p_n\,\log p_n
\end{equation}
is dissipating along the solution of the mean-field system
\begin{equation}\label{eq:standard_BDY}
  p'_n = \left\{
    \begin{array}{ll}
      p_1-(1-p_0)\,p_0 & \quad \text{for~~} n=0, \\
      p_{n+1}+(1-p_0)\,p_{n-1}- (1+(1-p_0))\,p_n & \quad \text{for~~} n \geq 1.
    \end{array}
  \right.
\end{equation}
Moreover, the geometric distribution with mean $\mu$ is the unique (global) minimizer of the Boltzmann functional $\HH[{\bf p}]$ among ${\bf p} \in V_\mu$.

We propose to adapt $\HH$ to our dynamics and employ the following candidate Lyapunov functional:
\begin{equation}\label{eq:naive_Lyapunov}
\mathcal{H}_{ab}[{\bf p}] \coloneqq \sum\limits_{n=0}^{a-1} p_n + \sum\limits_{n=a}^{b} p_n\,\log p_n + \sum\limits_{n=b+1}^{\infty} p_n.
\end{equation}
Indeed, if we take into account the fact that $p^*_n = 0$ for all $n \notin \{a,a+1,\ldots,b\}$ and that $\{p^*_n\}_{a\leq n\leq b}$ resembles a geometric distribution (restricted to a possibly bounded domain), it is natural to consider the local/truncated entropy of the form $\sum_{n=a}^{b} p_n\,\log p_n$. Meanwhile, since the time derivative of the local entropy $\sum_{n=a}^{b} p_n\,\log p_n$ will produce the term $\sum_{n=a}^{b} p'_n$ which does not have a definite sign, we compensate for the contribution of the aforementioned term by including the sum $\sum_{n \notin \{a,a+1,\ldots,b\}} p_n$ in the design of $\mathcal{H}_{ab}[{\bf p}]$.

A routine application of the method of Lagrange multipliers shows that the equilibrium distribution ${\bf p}^*$ \eqref{eq:equi} enjoys the following variational characterization:
\[{\bf p}^* = \argmin\limits_{{\bf p} \in V_\mu} \mathcal{H}_{ab}[{\bf p}].
\]
Let us investigate how the proposed Lyapunov functional $\mathcal{H}_{ab}[{\bf p}]$ behaves along the solution ${\bf p}(t)$ of the nonlinear ODE system \eqref{eq:law_limit}-\eqref{eq:two-classes}.
\begin{proposition}\label{ppo:decay_H_ab}
Assume that ${\bf p}(t)=\{p_n(t)\}_{n\geq 0}$ is a classical solution of \eqref{eq:law_limit}-\eqref{eq:two-classes} such that ${\bf p}(0) \in \mathcal{P}(\mathbb N)$ with mean $\mu \in \mathbb N_+$, then
\begin{equation}
\label{eq:decay_H_ab}
\frac{\dd}{\dd t} \mathcal{H}_{ab}[{\bf p}]  \leq - \lambda_r \log \lambda_r \sum_{n=b+1}^{\infty} p_n - \lambda_g \log \lambda_g \sum_{n=0}^{a-1} p_n,
\end{equation}
with $\lambda_r$ and $\lambda_g$ defined as in \eqref{eq:two-classes}.
\end{proposition}
\begin{proof}
A straightforward computation yields that
\begin{eqnarray*}
\frac{\dd}{\dd t} \mathcal{H}_{ab}[{\bf p}]
&=& \sum_{n=0}^{a-1}p'_n + \sum_{n=a}^{b} (p'_n\,\log p_n + p'_n) + \sum_{n=b+1}^{\infty}p'_n\\
&=& \sum_{n=a}^{b} p'_n\,\log p_n,
\end{eqnarray*}
where we have exploited the fact that $\sum_{n\geq 0} p_n'=0$ due to the preservation of the total probability mass. Therefore,
\begin{eqnarray*}
\frac{\dd}{\dd t} \mathcal{H}_{ab}[{\bf p}] &=&  \underbrace{\lambda_g\,p_{a-1}\,\log p_a + \lambda_r\,p_{b+1}\,\log p_b}_{\leq 0}\\
&\quad& +\displaystyle \sum_{n=a}^{b-1} \left[\lambda_rp_{n+1}-\lambda_gp_n\right]\log\frac{p_n}{p_{n+1}}.
\end{eqnarray*}
Since $p_{a-1}\,\log p_a\leq 0$ and $p_{b+1}\,\log p_b\leq 0$, we arrive at
\begin{eqnarray*}
\frac{\dd}{\dd t} \mathcal{H}_{ab}[{\bf p}] &\leq& -\sum_{n=a}^{b-1} \underbrace{\left[\lambda_r\,p_{n+1}-\lambda_g\,p_n\right]\log\frac{\lambda_r\,p_{n+1}}{\lambda_g\,p_n}}_{\geq 0} + \sum_{n=a}^{b-1} \left[\lambda_r\,p_{n+1}-\lambda_g\,p_n\right]\log\frac{\lambda_r}{\lambda_g}.
\end{eqnarray*}
By virtue of the elementary inequality $(y-x)\,(\log y -\log x) \geq 0$ for all $x,y\geq 0$, we conclude that
\begin{eqnarray*}
\frac{\dd}{\dd t} \mathcal{H}_{ab}[{\bf p}] &\leq& \underbrace{\sum_{n=0}^{a-1} \lambda_gp_n \log \lambda_r +\sum_{n=b+1}^{\infty} \lambda_rp_n \log \lambda_g}_{\leq 0} \\
&\qquad& - \sum_{n=0}^{a-1} \lambda_gp_n \log \lambda_g - \sum_{n=b+1}^{\infty} \lambda_rp_n \log \lambda_r,
\end{eqnarray*}
where we have used that $\log \lambda_r<0$ and $\log \lambda_g<0$. Thus the proof is completed.
\end{proof}
Unfortunately, for a general initial datum ${\bf p} \in \mathcal{A}$, the candidate Lyapunov functional $\mathcal{H}_{ab}[{\bf p}]$ does not necessarily decrease along solutions of the mean-field ODE system \eqref{eq:law_limit}-\eqref{eq:two-classes}. Nevertheless, one might conjecture that its time derivative becomes non-positive after some finite time.

In order to overcome the issue with the naive choice $\mathcal{H}_{ab}[{\bf p}]$, we decide to modify the functional $\mathcal{H}_{ab}$ by adding additional dissipating terms:
\begin{equation}\label{eq:energy_functional}
\widetilde{\mathcal{H}}_{ab}[{\bf p}] \coloneqq  \mathcal{H}_{ab}[{\bf p}] + k_1\sum\limits_{n=b+1}^{\infty} n\,p_n + k_2\sum\limits_{n=0}^{a-1} (a-n)\,p_n,
\end{equation}
The main idea is that with sufficiently large positive weights $k_1$ and $k_2$, the terms $\sum_{n\geq b+1} n\,p_n$ and $\sum_{n\leq a-1} (a-n)\,p_n$ will compensate for the possibly positive terms appearing in the time derivative of our modified candidate Lyapunov functional \eqref{eq:energy_functional}. We also emphasize that our generalized Boltzmann's $\widetilde{\mathcal{H}}_{ab}$ functional \eqref{eq:energy_functional} boils down to the classical $\HH$ functional \ref{eq:H_functional} in the special case when $a = 0$ and $b \to \infty$. As for our previous functional $\mathcal{H}_{ab}[{\bf p}]$, ne can verify via the method of Lagrange multipliers that the truncated geometric-type distribution ${\bf p}^*$ remains the unique global minimizer of $\widetilde{\mathcal{H}}_{ab}[{\bf p}]$:
\begin{equation*}
{\bf p}^* = \argmin\limits_{{\bf p} \in V_\mu} \widetilde{\mathcal{H}}_{ab}[{\bf p}]
\end{equation*}
for each fixed choice of the pair $(k_1,k_2) \in \mathbb{R}^2_+$. This can be regarded as a natural generalization of the minimum entropy characterization of the geometric distribution among ${\bf p} \in V_\mu$ (here “minimum” is used instead of “maximum” since the Shannon entropy $-\HH[{\bf p}]$ differs from the Boltzmann entropy $\HH[{\bf p}]$ by a minus sign).


We are now ready to state the main result of the manuscript.

\begin{theorem}[Dissipation of the generalized entropy]\label{thm:dissipation_of_E}
Let ${\bf p}(t) = \{p_n(t)\}_{n\geq 0}$ be a classical solution to the system of nonlinear ODEs \eqref{eq:law_limit}-\eqref{eq:two-classes} with ${\bf p}(0) \in \mathcal{A}$. There exist fixed $k_1>0$ and $k_2>0$ such that for all $t \geq 0$ we have
\begin{equation}\label{eq:dissipation_E}
\frac{\dd}{\dd t} \widetilde{\mathcal{H}}_{ab}[{\bf p}] \leq 0.
\end{equation}
Moreover, $\frac{\dd}{\dd t} \widetilde{\mathcal{H}}_{ab}[{\bf p}] = 0$ if and only if ${\bf p}$ coincides with ${\bf p}^*$ \eqref{eq:equi}. Consequently, ${\bf p}(t) \xrightarrow{t\to \infty} {\bf p}^*$ strongly in $\ell^p$ for each $1 < p < \infty$.
\end{theorem}
Explicit values for $k_1$ and $k_2$ can be determined and will be used in the proof of Theorem~\ref{thm:dissipation_of_E}. Specifically, we set
\begin{equation}\label{eq:k1k2}
k_1=-\ln{\left(1-\mu / b \right)} \quad \text{ and }\quad k_2= \ln\left( \max \left\{\frac{2^{b+1}+2^a}{1-\mu / b}, \sum_{n \geq 0} 2^np_n(0) \right\}\right).
\end{equation}

The rest of this subsection is devoted to the proof of Theorem \ref{thm:dissipation_of_E}, and to facilitate the presentation we first develop several preliminary results. We will need uniform (in time) lower bounds on $\lambda_r$ and $\lambda_g$, respectively. The uniform lower bound on $\lambda_r$ is rather elementary to establish, and the uniform lower bound on $\lambda_g$ requires additional effort along with a well-prepared initial datum\eqref{eq:IC_condition}.

\begin{lemma}\label{lem:1}
For each ${\bf p} \in \mathcal{A}$, let $C_0 = \sum_{n\geq 0} 2^n\,p_n < \infty$. Then
\begin{equation}\label{eq:uniform_lower_bounds}
\lambda_r \geq 1 - \frac{\mu}{b} \quad \textrm{and} \quad \lambda_g \geq 1/C_0.
\end{equation}
\end{lemma}

\Proof The lower bound on $\lambda_r$ follows essentially from Markov's inequality. Indeed, since $\mu = \sum_{n\geq 0} n\,p_n \geq \sum_{n\geq b} n\,p_n \geq b\,\sum_{n\geq b} p_n = b\,(1-\lambda_r)$, we deduce that $\lambda_r \geq 1 - \frac{\mu}{b}$. To obtain a lower bound on $\lambda_g$ we resort to the second moment method. Let $X \sim {\bf p}$ be a $\mathbb N$-valued random variable distributed according to ${\bf p}$, we have
\begin{equation*}
\mu = \mathbb{E}[X] = \mathbb{E}[X\,\mathbbm{1}\{X\geq \mu\}] + \mathbb{E}[X\,\mathbbm{1}\{X < \mu\}] \leq \sqrt{\mathbb{E}[X^2]}\,\sqrt{\mathbb{P}(X\geq \mu)} + \mu - 1,
\end{equation*}
where the inequality follows the observation that $X \in \mathbb N$ and $\mu \in \mathbb N_+$ imply $X\,\mathbbm{1}\{X < \mu\} \leq \mu-1$ almost surely. As $\mathbb{E}[X^2] = \sum_{n\geq 0} n^2\,p_n \leq \sum_{n\geq 0} 2^n\,p_n = C_0$ and $\mathbb{P}(X\geq \mu) \leq \lambda_g$, we deduce from the previous inequality that $\lambda_g\geq 1/C_0$, which completes the proof. \qed

From the content of Lemma \ref{lem:1}, a uniform lower bound on $\lambda_g$ can be hoped for once we establish a uniform (upper) bound for $\sum_{n\geq 0} 2^n\,p_n$, to which we now turn.

\begin{lemma}\label{lem:2}
Let ${\bf p}(t) = \{p_n(t)\}_{n\geq 0}$ be a classical solution to the system of nonlinear ODEs \eqref{eq:law_limit}-\eqref{eq:two-classes} with ${\bf p}(0) \in V_\mu$, we have
\begin{equation}\label{eq:uniform_exponential_moment}
\sup\limits_{t\geq 0} \sum\limits_{n\geq 0} K^n\,p_n(t) \leq \max\left\{\frac{K^{b+1}+K^a}{1-\mu/b},\sum_{n\geq 0} K^n\,p_n(0)\right\}
\end{equation}
for all $K > 1$, provided that $\sum_{n\geq 0} K^n\,p_n(0) < \infty$.
\end{lemma}

\Proof Let $y(t) = \sum_{n\geq 0} K^n\,p_n(t)$. A straightforward computation of the time derivative of $y$ gives rise to
\begin{align*}
y' &= \lambda_g\,(K-1)\,\sum_{n\leq b-1} K^n\,p_n + \lambda_r\,(1/K - 1)\,\sum_{n\geq a+1} K^n\,p_n \\
&= \lambda_g\,(K-1)\,\sum_{n\leq b-1} K^n\,p_n + \lambda_r\,(1- 1/K)\,\sum_{n\leq a} K^n\,p_n - \lambda_r\,(1- 1/K)\,y \\
&\leq (K-1)\,K^b + (1- 1/K)\,K^a - \left(1-\frac{\mu}{b}\right)\,(1- 1/K)\,y.
\end{align*}
Consequently, we deduce that $y' \leq 0$ as long as
\[y \geq \frac{(K-1)\,K^b + (1- 1/K)\,K^a}{(1- 1/K)(1-\mu/b)} = \frac{K^{b+1}+K^a}{1-\mu/b},\]
from which the announced uniform bound \eqref{eq:uniform_exponential_moment} follows. \qed

\begin{remark}
It is worth emphasizing that the strategy employed in the proof of Lemma \ref{lem:2} will break down when $b \to \infty$ (as $K^{b+1}$ blows up when $b \to \infty$), and a more delicate argument will be required in order to establish a analogue of the uniform bound \eqref{eq:uniform_exponential_moment} in the asymptotic region as $b \to \infty$ (see for instance the recent work \cite{cao_interacting_2022} on the classical BDY model corresponding to $a = 0$ and $b \to \infty$).
\end{remark}

We also need the following elementary result from calculus.
\begin{lemma}\label{lem:3}
Let $\gamma \in (0,1)$, $k > 0$, and define $f(x) \coloneqq -x\,\ln x - k\,x$ for $x \in \mathbb{R}_+$. Then $f(x) \leq 0$ for all $x \in [\gamma, 1]$ if $k \geq -\ln \gamma$.
\end{lemma}

\Proof It suffices to notice that $f(\gamma) = -\gamma\,\ln \gamma - k\,\gamma \leq 0$ and $f'(x) = -\ln x - 1-k \leq 0$ for $x \in [\gamma, 1]$, whenever $k\geq -\ln \gamma$. \qed

Armed with the previous set of preliminary lemmas, we are now in a position to present the proof of Theorem \ref{thm:dissipation_of_E}.\\

\noindent{\bf Proof of Theorem \ref{thm:dissipation_of_E}} We first invoke the computation in Proposition \ref{ppo:decay_H_ab} to obtain the time derivative of the new functional $\widetilde{\mathcal{H}}_{ab}[{\bf p}]$:
\begin{eqnarray}
  \frac{\dd}{\dd t} \widetilde{\mathcal{H}}_{ab}[{\bf p}]  &\leq& - \lambda_r \log \lambda_r \sum_{n=b+1}^{\infty} p_n - \lambda_g \log \lambda_g \sum_{n=0}^{a-1} p_n \\
&+& k_1\sum\limits_{n=b+1}^{\infty} n\,p_n' + k_2\sum\limits_{n=0}^{a-1} (a-n)\,p_n'.
\end{eqnarray}
Next, we notice that for $0\leq n\leq a-1$, $\sum_{\ell=0}^n p'_\ell = -\lambda_g\,p_n$. Therefore,
\begin{equation}\label{eq:Term1}
\sum\limits_{0\leq n\leq a-1} (a-n)\,p'_n = \sum\limits_{n=0}^{a-1} \sum_{\ell=0}^n p'_\ell = -\lambda_g\,\sum\limits_{0\leq n\leq a-1} p_n \leq 0.
\end{equation}
On the other hand, we also have
\begin{equation}\label{eq:Term2}
\sum\limits_{n\geq b+1} n\,p'_n = -\lambda_r\,\left[\sum\limits_{n\geq b+1} p_n + b\,p_{b+1}\right] \leq  -\lambda_r\,\sum\limits_{n\geq b+1} p_n \leq 0.
\end{equation}
Assembling these estimates together leads us to
\begin{equation}\label{eq:derivative_of_E}
\begin{aligned}
\frac{\dd}{\dd t} \widetilde{\mathcal{H}}_{ab}[{\bf p}] &\leq \left[-\lambda_r\,\log \lambda_r-k_1\,\lambda_r\right]\,\sum_{n\geq b+1} p_n + \left[-\lambda_g\,\log \lambda_g-k_2\,\lambda_g\right]\,\sum_{n\leq a-1} p_n \\
&\leq 0,
\end{aligned}
\end{equation}
where the last inequality follows from Lemmas \ref{lem:1}-\ref{lem:3} together with the definitions of $k_1$ and $k_2$ \eqref{eq:k1k2}.

Moreover, it is not hard to verify that $\frac{\dd}{\dd t} \widetilde{\mathcal{H}}_{ab}[{\bf p}] = 0$ only when ${\bf p} = {\bf p}^*$. Finally, the large time convergence guarantee stated in Theorem \ref{thm:dissipation_of_E} can be shown in a manner similar to the standard mead-field BDY dynamics \eqref{eq:standard_BDY} (see for instance \cite{merle_cutoff_2019,cao_derivation_2021}). \qed

To investigate numerically the convergence of ${\bf p}(t)$ its equilibrium distribution ${\bf p}^*$, we set $μ=7$, $a = 5$ and $b = 10$ and use $1001$ components to approximate the distribution ${\bf p}(t)$ (i.e., ${\bf p}(t) \approx (p_0(t),…,p_{1000}(t))$). As an initial condition, we use $p_μ(0)=1$ and $p_i(0)=0$ for $i \neq μ$. The standard Runge-Kutta fourth-order method is used to discretize the mean-field ODE system \eqref{eq:law_limit}-\eqref{eq:two-classes} with the time step $Δt=0.01$.

We plot in Figure \ref{fig:conv_equilibrium}-left the numerical solution ${\bf p}$ for various $0\leq t\leq 5$ units of time and compare them with the equilibrium distribution ${\bf p}^*$. It is worth highlighting that the two distributions ${\bf p}(t=5)$ and ${\bf p}^*$ are almost indistinguishable. Indeed, as shown in Figure \ref{fig:conv_equilibrium}, the convergence toward the equilibrium distribution is exponential in time.


\begin{figure}[ht]
    \centering
    \includegraphics[width=.55\textwidth]{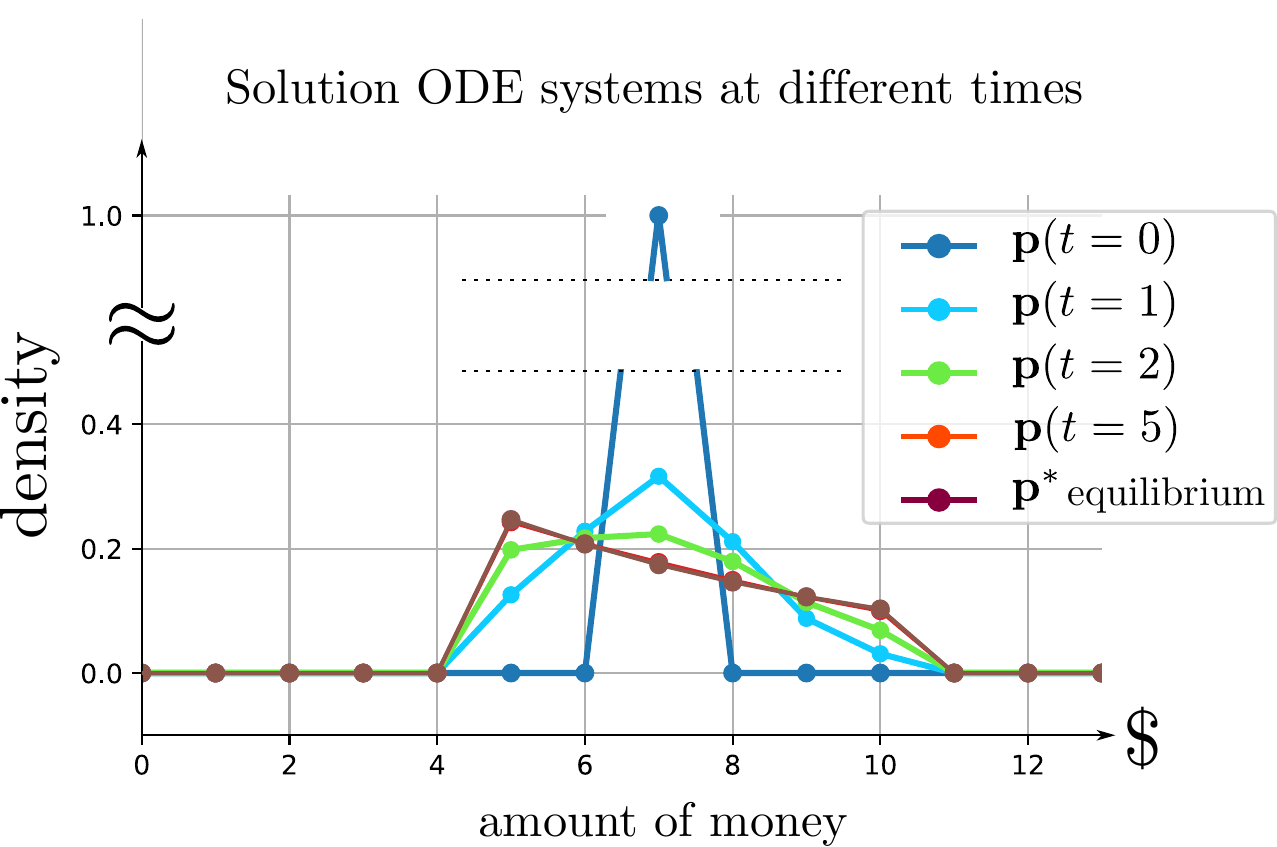}
    \includegraphics[width=.4\textwidth]{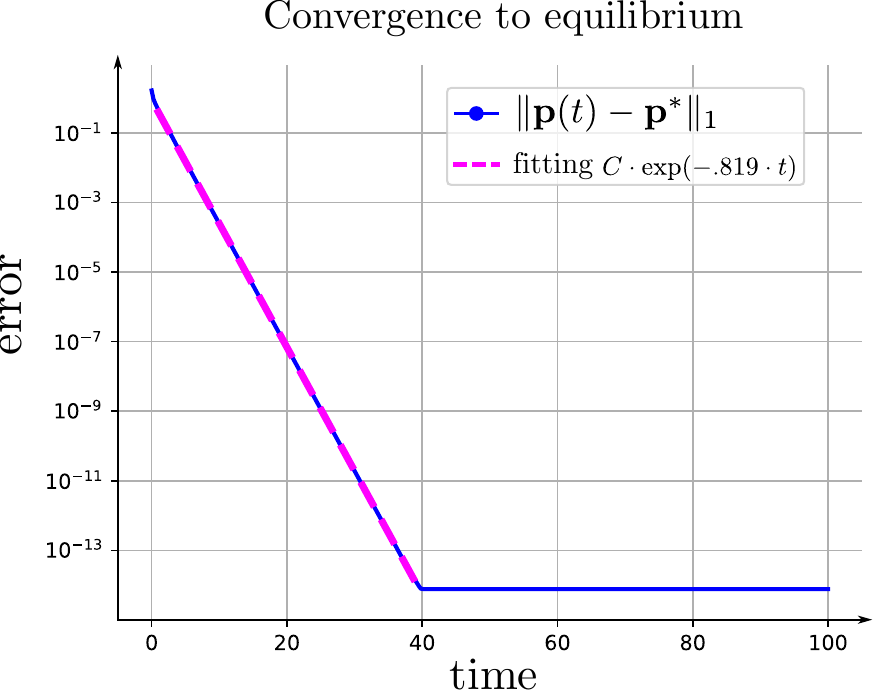}
    \caption{\textbf{Left:} Comparison between the numerical solution ${\bf p}(t)$ of the mean-field model \eqref{eq:law_limit}-\eqref{eq:two-classes} and the geometric-type equilibrium ${\bf p}^*$ for $0\leq t\leq 5$. \textbf{Right:} Exponential decay of the distance between ${\bf p}(t)$ toward the equilibrium ${\bf p}^*$.}
    \label{fig:conv_equilibrium}

\end{figure}

\subsection{Gini index comparison and mitigation of inequality}\label{subsec:3.2}

We would like to examine the impact of wealth flooring and ceiling on the wealth inequality of the distribution ${\bf p}$ measured by the Gini index, which is defined for any ${\bf p} \in V_\mu$ by
\begin{equation}\label{def:Gini}
G[{\bf p}] = \frac{1}{2\,\mu}\,\sum\limits_{i,j\geq 0} |i-j|\,p_i\,p_j.
\end{equation}
The Gini index is a widely used inequality index which ranges from $0$ (perfect equality) to $1$ (maximum inequality). In addition to its applications in econophysics models \cite{boghosian_h_2015,cao_derivation_2021,cao_explicit_2021}, the notion of Gini index has also found its application in many other fields beyond economics, such as sociophysics (opinion dynamics) \cite{cao_iterative_2024}, sparse representation of signals \cite{hurley_comparing_2009}, and the so-called dispersion process on complete graphs \cite{cao_sticky_2024}.

It is natural to expect that introducing a wealth floor or ceiling reduces wealth inequality relative to the standard BDY model. Formally, this intuition leads to the following conjecture, referred to as the \textbf{Gini index comparison conjecture}: given an initial condition ${\bf p}(0) \in V_\mu$, let ${\bf p}^{a,b}(t)$ denote the solution to the mean-field ODE system \eqref{eq:law_limit}-\eqref{eq:two-classes} at time $t$. Then for all $t > 0$, the Gini index $G[{\bf p}^{a,b}(t)]$ is non-increasing in $a < \mu$ and non-decreasing in $b > \mu$. To examine this conjecture numerically, we plot the evolution of the Gini index $G[{\bf p}^{a,b}]$ for various parameters $a$ and $b$, while keeping the average amount of money $\mu = 5$ fixed (see Figure \ref{fig:gini}). It is evident that the above conjecture is supported by numerical evidence.

Although we do not have a rigorous proof of the aforementioned Gini index comparison conjecture (which we will leave as an open problem for future research endeavors), its validity can be tested in the special case as $b \to \infty$ and $a\in \mathbb N$, when the solution ${\bf p}^{a,b}$ reaches its stationarity.
\begin{proposition}\label{prop:Gini_comparison_special_case}
Let $a \in \mathbb N$ and $b \to \infty$. Then the equilibrium solution $({\bf p}^{a,\infty})^*$ of the mean-field system \eqref{eq:law_limit} with ${\bf p}(0) \in V_\mu$ satisfies the following Gini index comparison principle
\begin{equation}\label{eq:Gini_comparison_special_case}
G[({\bf p}^{a',\infty})^*] \leq G[({\bf p}^{a,\infty})^*]
\end{equation}
whenever $a \leq a' < \mu$.
\end{proposition}

\Proof The key ingredient in the proof is the completely explicit expression of the equilibrium distribution $({\bf p}^{a,\infty})^*$ provided by \eqref{eq:equi_b=infty}. After a routine but tedious computation, we obtain
\begin{equation}\label{eq:Gini_explicit}
G[({\bf p}^{a,\infty})^*] = 1 - \frac{\mu^2 - a^2 + a}{\mu\,[1+2\,(\mu-a)]}.
\end{equation}
It is worth noticing that in the special case where $a=0$, \eqref{eq:Gini_explicit} gives us $G[({\bf p}^{0,\infty})^*] = \frac{1+\mu}{1+2\mu}$, which coincides with the Gini index of a geometric distribution whose success probability equal to $1 - ({\bf p}^{0,\infty})^*_0 = \frac{\mu}{1+\mu}$. To finish the proof, it suffices to note that the map $a \mapsto \frac{\mu^2 - a^2 + a}{\mu\,[1+2(\mu-a)]}$ is a non-decreasing function for all $0 \leq a < \mu$. \qed

\begin{figure}[ht]
    \centering
    \includegraphics[width=.95\textwidth]{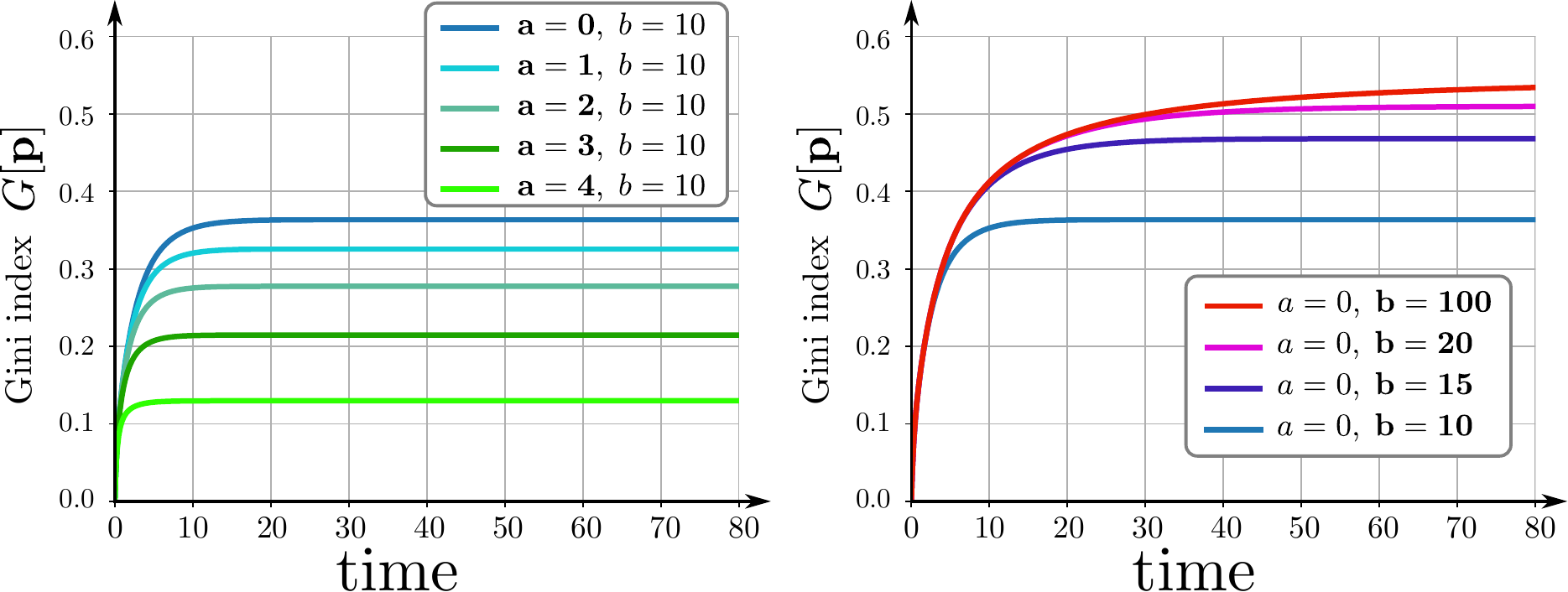}
    \caption{Evolution of the Gini index for different values of $a$ and $b$. As the wealth floor $a$ increases (left panel), the gap between agents decreases, leading to a reduction in the Gini index. Conversely, increasing the wealth ceiling $b$ accentuates inequality, resulting in a substantial rise in the Gini index (right panel). Parameters: $\Delta t=10^{−1}$ and $\mu=5$.}
    \label{fig:gini}

\end{figure}

\section{Conclusion}\label{sec:4}
\setcounter{equation}{0}

In this work, we introduce and analyze a natural extension of the classical Bennati–Dragulescu–Yakovenko (BDY) binary exchange model \cite{dragulescu_statistical_2000}, incorporating both wealth flooring and ceiling effects. Although these additional constraints complicate the analysis at the agent-based level, we rigorously derive a corresponding mean-field system of ODEs that facilitates the study of its long-time behavior through a suitably constructed entropy-like Lyapunov functional. We further examine how the parameters $a$ and $b$ influence wealth inequality. In particular, we conjecture, supported by numerical evidence presented in section \ref{subsec:3.2}, that imposing wealth floor and ceiling constraints mitigates inequality, as measured by the Gini index, compared to the original BDY model. However, a rigorous proof of this conjecture remains open.

To the best of our knowledge, no mean-field analysis has been developed for the present model prior to this study. The results obtained here open several promising avenues for further research. A natural next step would be to establish quantitative decay estimates for the generalized entropy $\widetilde{\mathcal{H}}_{ab}[{\bf p}]$ at the mean-field level, thereby providing a sharper characterization of convergence to equilibrium. Another important direction is to rigorously verify the monotonicity of the Gini index of ${\bf p}(t)$ with respect to the parameters $a$ and $b$ for a fixed mean wealth $\mu$. Progress on these fronts would yield a deeper and more quantitative understanding of how wealth floor and ceiling constraints shape the long-term behavior and inequality structure of the system.

Finally, we emphasize that econophysics models incorporating wealth floor and ceiling mechanisms have received relatively little attention in recent years, despite their clear relevance to real-world economic policy. In practice, wealth floors can be interpreted as social safety nets or minimum-income guarantees, while wealth ceilings may reflect progressive taxation or redistribution policies designed to limit excessive accumulation. From this perspective, the model investigated here provides a simple yet insightful theoretical framework for exploring how such constraints influence collective wealth dynamics and wealth inequality. Our results suggest that even minimal forms of floor and ceiling enforcement can significantly affect the long-term distributional outcomes, motivating a more systematic study of constrained exchange models as a bridge between statistical physics and socio-economic policy design. \\

\noindent {\bf Acknowledgement~} Fei Cao gratefully acknowledges support from an AMS-Simons Travel Grant, administered by the American Mathematical Society with funding from the Simons Foundation.

\appendix

\section{Appendix: Proof of Theorem~\ref{thm:PoC}}\label{sec:Appendix}

We aim to establish the quantitative and finite-time propagation of chaos estimate \eqref{eq:PoC}, building on techniques developed in recent works on the classical BDY dynamics and its biased version \cite{cao_derivation_2021,merle_cutoff_2019}. First, we notice that when agent $i$ (with $S_i > a$)  gives a dollar to $j$ (with $S_j < b$), the empirical measure experiences the following transform:
\begin{equation}
  \label{eq:jump_emp}
  {\bf p}_{\emp}\;\;\begin{tikzpicture} \draw [->,decorate,decoration={snake,amplitude=.4mm,segment length=2mm,post length=1mm}]
    (0,0) -- (.6,0);\end{tikzpicture}\;\; {\bf p}_{\emp} + \frac{1}{N}\Big(δ_{S_i-1} + δ_{S_j+1}- δ_{S_i}- δ_{S_j}\Big) \quad (\text{if } S_i > a ~\text{and}~ S_j < b).
\end{equation}
Due to the settings of the agent-based model \eqref{dynamics:modified_BDY}, the evolution of the empirical law ${\bf p}_{\emp}$ is governed by the following stochastic differential equation (SDE):
\begin{equation}\label{eq:empirical_SDE}
\dd {\bf p}_{emp}(t) = \frac{1}{N}\,∑_{k=1,\ell=0}^{+∞} \left(δ_{k-1} + δ_{\ell+1}- δ_{k}- δ_{\ell}\right) \dd \mathrm{N}_t^{(k,\ell)},
\end{equation}
in which $\{\mathrm{N}_t^{(k,\ell)}\}_{1 \leq k,\ell \leq N}$ denotes a
collection of independent Poisson clocks with intensity
\begin{equation}\label{eq:lambda_kl}
λ^{(k,\ell)} = p_{\emp,k}\,\left(N\,\rho_{\emp,\ell}-\mathbbm{1}_{\{k=\ell\}}\right)\,\mathbbm{1}_{\{k \geq a+1\}}\,\mathbbm{1}_{\{\ell \leq b-1\}}.
\end{equation}
We claim that the infinitesimal generator $\mathcal{D} \colon \ell^1(\mathbb N) \to \ell^1(\mathbb N)$ associated to the modified BDY model \eqref{dynamics:modified_BDY} is given by
\begin{equation}\label{eq:generator}
\mathcal{D} = \mathcal{L} + \frac{1}{N}\,\mathcal{R},
\end{equation}
where $\mathcal{L}$ is the nonlinear operator appearing in the mean-field ODE system \eqref{eq:law_limit} and the operator $\mathcal{R}$ is defined by
\begin{equation}\label{eq:R}
\mathcal{R}[{\bf p}]_k \coloneqq 2\,p_k\,\mathbbm{1}_{\{a+1 \leq n\leq b-1\}} - p_{k-1}\,\mathbbm{1}_{\{a+2 \leq n\leq b\}} - p_{k+1}\,\mathbbm{1}_{\{a \leq n\leq b-2\}}
\end{equation}
for each $k\in \mathbb N$. To see \eqref{eq:generator}, it suffices to notice that for any continuous and bounded test function $\varphi$, we have
\begin{equation*}
\begin{aligned}
&\frac{\dd}{\dd t} \mathbb{E}\left[\langle {\bf p}_{\emp}, \varphi\rangle\right] = \mathbb{E}\left[\sum_{k=a+1}^\infty \sum_{\ell=0}^{b-1} \left(\varphi(k-1)+\varphi(\ell+1)-\varphi(k)-\varphi(\ell)\right)\,{\bf p}_{\emp,k}\,{\bf p}_{\emp,\ell}\right] \\
&\qquad \qquad \qquad \qquad \qquad -\frac{1}{N}\,\mathbb{E}\left[\sum_{k=a+1}^{b-1} \left(\varphi(k-1) + \varphi(k+1) - 2\,\varphi(k)\right)\,{\bf p}_{\emp,k}\right] \\
&= \sum_{\ell=0}^{b-1} {\bf p}_{\emp,\ell}\,\mathbb{E}\left[\sum_{k=a+1}^\infty \left(\varphi(k-1)-\varphi(k)\right)\,{\bf p}_{\emp,k}\right] + \sum_{k=a+1}^\infty {\bf p}_{\emp,k}\,\mathbb{E}\left[\sum_{\ell=0}^{b-1} \left(\varphi(\ell+1)-\varphi(\ell)\right)\,{\bf p}_{\emp,\ell}\right] \\
&\qquad -\frac{1}{N}\,\mathbb{E}\left[\sum_{k=a+1}^{b-1} \left(\varphi(k-1) + \varphi(k+1) - 2\,\varphi(k)\right)\,{\bf p}_{\emp,k}\right] \\
&= \mathbb{E}\left[\langle \mathcal{L}[{\bf p}_{\emp}], \varphi\rangle\right] + \frac{1}{N}\,\mathbb{E}\left[\langle \mathcal{R}[{\bf p}_{\emp}], \varphi\rangle\right] = \mathbb{E}\left[\langle \mathcal{D}[{\bf p}_{\emp}], \varphi\rangle\right].
\end{aligned}
\end{equation*}
Intuitively, the operator $\mathcal{R}$ \eqref{eq:R} captures the (average) deviation in the dynamics of the empirical measure ${\bf p}_{\emp}(t)$ from those of the mean-field limit
${\bf p}(t)$ governed by \eqref{eq:law_limit}-\eqref{eq:two-classes}. Such deviations diminish
as the population size $N$ grows large. Thanks to Dynkin's formula, the compensated process of the empirical measure ${\bf p}_{\emp}(t)$, defined by
\begin{equation}\label{eq:M}
M(t) \coloneqq {\bf p}_{\emp}(t) - {\bf p}_{\emp}(0) - \int_0^t \mathcal{D}[{\bf p}_{\emp}]\,\dd s
\end{equation}
is a $\ell^1(\mathbb N)$-valued martingale. We now control the size of this compensated martingale, for which we first establish the following preliminary result:

\begin{lemma}\label{lem:A1}
Assume that ${\bf x} = (x_0,x_1,\cdots) \in \ell^1(\mathbb N)$ satisfies $\sum_{k\geq 0} (k+1)\,|x_k| \leq C$ and $\sum_{k\geq 0} (k+1)\,|x_k|^2 \leq \frac{C}{N}$ for some constants $C > 0$ and $N \geq 1$. Then there exists a universal constant $L > 0$ (independent of $N$ and ${\bf x}$) such that
$\|x\|_{\ell^1} \leq \frac{L\,\sqrt{\log N}}{\sqrt{N}}$.
\end{lemma}

\begin{proof}
For each $M\geq 1$, we apply the Cauchy-Schwarz inequality to obtain
\begin{equation*}
\begin{aligned}
\sum_{k\geq 0} |x_k| &= \sum_{k=0}^M |x_k| + \sum_{k=M+1}^\infty |x_k| \\
&= \left(\sum_{k=0}^M \frac{1}{k+1}\right)^{\frac 12}\left(\sum_{k=0}^M (k+1)\,|x_k|^2\right)^{\frac 12} + \frac{1}{M+2}\,\sum_{k=M+1}^\infty (k+1)\,|x_k| \\
&\leq \left(\sum_{k=1}^{M+1} \frac{1}{k}\right)^{\frac 12} \left(\sum_{k\geq 0}(k+1)\,|x_k|^2\right)^{\frac 12} + \frac{1}{M+2}\,\sum_{k\geq 0} (k+1)\,|x_k| \\
&\leq \left(\frac{C}{N}\right)^{\frac 12}\left(\sum_{k=1}^{M+1} \frac{1}{k}\right)^{\frac 12} + \frac{C}{M+2}.
\end{aligned}
\end{equation*}
Therefore, taking $M$ to be the integer part of $N$ and using the bound $\sum_{k=1}^{M+1} \tfrac{1}{k} \leq 1+\log (M+2)$,  we arrive at the desired estimate.
\end{proof}

\begin{proposition}\label{prop:martingale_control}
Let $M(t)$ \eqref{eq:M} be the compensated process of the empirical measure ${\bf p}_{\emp}(t)$. There exists some universal constant $C > 0 $ such that
\begin{equation}\label{eq:martingale_bound}
\mathbb{E}\left[\|M(t)\|_{\ell^1}\right] \leq \frac{C\,t\,\sqrt{\log N}}{\sqrt{N}}~~\text{for all $t\geq 0$.}
\end{equation}
\end{proposition}

\begin{proof}
We note that the $k$-th component $M_k$ forms a continuous-time martingale,
whose jumps have magnitude at most $2/N$ and occur at rate bounded by
$N \left( 2\,{\bf p}_{\emp,k} + {\bf p}_{\emp,k-1} + {\bf p}_{\emp,k+1} \right)\,\dd t$. Therefore,
\begin{equation*}
\begin{aligned}
\mathbb{E}[|M_k(t)|] &\leq \frac{2}{N}\,\int_0^t N \left( 2\,{\bf p}_{\emp,k}(s) + {\bf p}_{\emp,k-1}(s) + {\bf p}_{\emp,k+1}(s) \right)\,\dd s,\\
\mathbb{E}[|M_k(t)|]^2 &\leq \frac{4}{N}\,\int_0^t N \left( 2\,{\bf p}_{\emp,k}(s) + {\bf p}_{\emp,k-1}(s) + {\bf p}_{\emp,k+1}(s) \right)\,\dd s.
\end{aligned}
\end{equation*}
Thanks to the conservation of total wealth, expressed by the fact that $\sum_{k\geq 0} k\,{\bf p}_{\emp,k}(t) = \mu$ for all $t\geq 0$, we deduce that
\begin{equation}\label{eq:M_estimates}
\sum_{k\geq 0} (k+1)\,\mathbb{E}[|M_k(t)|] \leq C\,t ~~\text{and}~~ \sum_{k\geq 0} (k+1)\,\mathbb{E}[|M_k(t)|^2] \leq \frac{C\,t}{N}
\end{equation}
for some universal constant $C > 0$. Finally, the desired estimate \eqref{eq:martingale_bound} follows readily from \eqref{eq:M_estimates} and Lemma \ref{lem:A1}.
\end{proof}

Next, a straightforward application of the triangle inequality shows that
the operator $\mathcal{L}$ is globally Lipschitz on
$\ell^1(\mathbb N)\cap \mathcal{P}(\mathbb N)$, and
$\mathcal{R}$ defines a bounded operator on $\ell^1(\mathbb N)$:
\begin{eqnarray}\label{eq:lip_L}
\|\mathcal{L}[{\bf p}]-\mathcal{L}[{\bf q}]\|_{\ell^1} &\leq& 8\,\|{\bf p}-{\bf q}\|_{\ell^1} \quad \text{for any } {\bf p},{\bf q}\in\ell^1(\mathbb N)\cap \mathcal{P}(\mathbb N)\\
\label{eq:bound_R}
\|\mathcal{R}[{\bf p}]\|_{\ell^1} &\leq& 4\,\|{\bf p}\|_{\ell^1} \!\quad\qquad \text{for any } {\bf p}\in \ell^1(\mathbb N)
\end{eqnarray}
To advance the proof of Theorem \ref{thm:PoC}, we first recall the evolution
equations satisfied by the empirical measure ${\bf p}_{\emp}$ and by
the mean-field law ${\bf p}(t)$, respectively:
 \begin{eqnarray*}
    {\bf p}(t) &=& {\bf p}(0) + \int_{0}^t \mathcal{L}[{\bf p}(s)]\,\dd s \\
    {\bf p}_{\emp}(t) &=& {\bf p}_{\emp}(0) + \int_{0}^t \mathcal{L}[{\bf p}_{\emp}(s)]\,\dd s + \frac{1}{N}\int_{0}^t \mathcal{R}[{\bf p}_{\emp}(s)]\,\dd s + M(t)
  \end{eqnarray*}
Setting $\mathcal{E}(t) = \mathbb{E}\left[\|{\bf p}(t) - {\bf p}_{\emp}(t)\|_{\ell^1}\right]$ and employing the estimates \eqref{eq:lip_L}--\eqref{eq:bound_R}, we obtain
\begin{align*}
\mathcal{E}(t) &\leq \mathcal{E}(0) + \int_{0}^t \mathbb{E}\left[\|\mathcal{L}[{\bf p}(s)]-\mathcal{L}[{\bf p}_{\emp}(s)]\|_{\ell^1}\right]\,\dd s +\frac{1}{N} ∫_{0}^t \mathbb{E}\left[\|\mathcal{R}[{\bf p}_{emp}(s)]\|_{\ell^1}\right]\,\dd s \\
&\qquad + \mathbb{E}\left[\|M(t)\|_{\ell^1}\right]\\
&\leq \mathcal{E}(0) + 8\int_{0}^t \mathbb{E}\left[\|{\bf p}_{emp}(s)-{\bf p}(s)\|_{\ell^1}\right]\,\dd s +\frac{4\,t}{N} + \mathbb{E}\left[\|M(t)\|_{\ell^1}\right] \\
&\leq \mathcal{E}(0) + 8\int_{0}^t \mathcal{E}(s)\,\dd s + \frac{4\,t}{N} + \frac{C\,t\,\sqrt{\log N}}{\sqrt{N}}.
\end{align*}
Consequently, the proof of Theorem \ref{thm:PoC} is completed thanks to a routine application of Gr\"{o}nwall's inequality.

\end{document}